\documentclass[12pt, reqno]{amsart}
\usepackage{amsmath, amsthm, amscd, amsfonts, amssymb, graphicx, color, mathrsfs}
\usepackage[bookmarksnumbered, colorlinks, plainpages]{hyperref}
\usepackage[all]{xy}
\usepackage{slashed}

\newtheorem{theorem}{Theorem}[section]
\newtheorem{lemma}[theorem]{Lemma}
\newtheorem{proposition}[theorem]{Proposition}

\theoremstyle{definition}
\newtheorem{definition}[theorem]{Definition}
\newtheorem{example}[theorem]{Example}

\theoremstyle{remark}
\newtheorem{remark}[theorem]{Remark}
\numberwithin{equation}{section}

\begin{document}
\setcounter{page}{1}

\title[ The index of elliptic operators on compact Lie groups]{ On the  index of pseudo-differential  operators on compact Lie groups}

\author[D. Cardona]{Duv\'an Cardona}
\address{
  Duv\'an Cardona:
  \endgraf
  Department of Mathematics  
  \endgraf
  Pontificia Universidad Javeriana.
  \endgraf
  Bogot\'a
  \endgraf
  Colombia
  \endgraf
  {\it E-mail address} {\rm duvanc306@gmail.com;
cardonaduvan@javeriana.edu.co}
  }


\dedicatory{}



\begin{abstract}
 In this note we study  the analytical index of  pseudo-differential operators by using the notion of (infinite dimensional) operator-valued symbols (in the sense of Ruzhansky and Turunen). Our main tools will be the McKean-Singer index formula together with the operator-valued functional calculus developed here.
\end{abstract} \maketitle

\tableofcontents

\section{Introduction}

In this note we investigate index formulae  for pseudo-differential operators on compact Lie groups by using the notion of operator-valued symbol.

A pseudo-differential operator $A:C^\infty_{0}(\mathbb{R}^n)\rightarrow C^\infty(\mathbb{R}^n), $ is an integral operator  defined by
\begin{equation}
Af(x)=\int_{\mathbb{R}^n} e^{ix\xi}\sigma_A(x,\xi)\hat{f}(\xi)\, d\xi,
\end{equation}
and associated to a smooth function $\sigma_A(x,\xi)$ -- called the symbol of $A$ --  satisfying some bounded conditions on its derivatives (see \cite{Hor2}).  Here $\hat{f}$ denotes the euclidean Fourier transform of the function $f$. For every $m\in \mathbb{R}$ and every open set $U\subset \mathbb{R}^n$, the H\"ormander class of symbols of order $m,$ $S^{m}(U\times \mathbb{R}^n)$, (for a detailed description see \cite{Hor2}),  is defined  by functions satisfying the usual estimates
\begin{equation}\label{horm}
|\partial_{x}^{\alpha}\partial_{\xi}^{\beta}\sigma_A(x,\xi)|\leq C_{\alpha,\beta}\langle \xi\rangle^{m-|\beta|},
\end{equation}
for all $(x,\xi) \in T^*U \cong U\times \mathbb{R}^{n}$ and $\alpha,\beta\in \mathbb{N}^n$. These classes  initially defined on open sets of  $\mathbb{R}^n,$ can be defined on smooth manifolds by using charts. On a manifold $M$ (orientable and without boundary), the corresponding operators associated to the H\"ormander classes of order $m$ will be denoted by $\Psi^m(M).$ In our case we are interested when $M=G$ is a compact Lie group.

It is well known that  every elliptic pseudo-differential operator $D$ on a closed manifold $M$ (i.e a compact manifold without boundary)  acting in smooth functions, has kernel and cokernel of finite dimension and, in terms of the $L^2$-theory of Fredholm operators, so it
can be associated an integer number, called the index of $D$ and, defined by
\begin{equation}
\textnormal{ind}(D):=\dim \textnormal{Kernel}(D)-\dim \textnormal{Cokernel}(D).
\end{equation}

Now, let $G$ be a compact Lie group, $\mathscr{D}(G)=C^\infty(G)$ be the space of smooth functions on $G$ endowed with the usual Frechet structure and $\mathscr{D}'(G)$ be the space of Schwartz distributions.  Let us consider a continuous operator $A:C^\infty(G)\rightarrow C^\infty(G) $ and the right convolution operator $r(f)$ on $G$ (defined by $r(f)(g)=g\ast f,$ where $f\in \mathscr{D}'(G),$  $g\in C^\infty(G)$). If $\pi_R$ is the  right regular representation on $G$ (defined  by $\pi_R(x)f(y)=f(yx),$ $x,y\in G, $  $f\in C^\infty(G)$) then, Ruzhansky and Turunen in \cite{Ruz} showed that
\begin{eqnarray}
Af(x)=\textnormal{tr}(\sigma_A(x)r(f)\pi_R(x)),
\end{eqnarray}
for some unique (operator-valued) symbol 
$\sigma_A:G\rightarrow\mathscr{B}(C^\infty(G))$ from $G$ into the space of continuous linear operators on $C^\infty(G).$ The pseudo-differential term is justified because $r(f),$ is sometimes, called the right global Fourier transform of $f.$

 In relation with our work, a characterization for the H\"ormander classes $\Psi^m(G)$ of pseudo-differential operators on arbitrary compact Lie groups,  in terms of operator-valued symbols, was proved by M. Taylor (see Remark 10.11.22 of \cite{Ruz}). 
If $A:C^\infty(G)\rightarrow C^\infty(G) $ extends to  a Fredholm operator on $L^2(G) $ (a necessary and sufficient condition for the Fredholmness of $A$ is the ellipticity condition), in this paper we compute the index of $A$  in terms of its operator-valued symbol $\sigma_A:G\rightarrow\mathscr{B}(C^\infty(G))$ and the operator-valued symbol of its formal adjoint $A^*,$ $\sigma_{A^*}:G\rightarrow\mathscr{B}(C^\infty(G)).$  Complete references on index theory are the books \cite{BB,Fedosovbook} and \cite{Fedosovbook2}.

The main result in  index theory is the Atiyah-Singer index theorem. This theorem was conjectured by I. M. Gelfand and several of its versions or extensions  can be found in the works of M. Atiyah, I. Singer and  R. Bott,  \cite{atiyabott1,atiyabott2,AS,AS1,AS2,AS3,AS44,AS4,AS5,BB,Bott} for several classes of manifolds (and non-commutative structures). It was proved in these references  that the index of an elliptic operator can be written  in topological terms depending only on the homotopy class of its principal symbol. For a general elliptic pseudo-differential operator $D$, acting on smooth sections of a closed manifold $M$, the Atiyah-Singer index formula has the following structure. First, the principal symbol of the operator $D$ defines a Chern character $\text{ch}(\sigma_D)$ which is a cohomology class with compact support in $TM$, the tangent bundle on $M$. In addition, there exists a cohomology class $\text{td}(TM\otimes \mathbb{C})$, called the Todd class, which give rise to the following integral expression for the index of $D$
\begin{equation}\label{as}
\text{ind}(D)=(-1)^n \int_{TM}\text{ch}(\sigma_D)\text{td}(TM\otimes \mathbb{C}),
\end{equation}
where $n=\dim(M)$, see \cite{AS1}. Although the general Atiyah-Singer index theorem applies for compact Lie groups, our main goal is to write the index of elliptic operators as an integral expression taking advantage of the Ruzhansky-Turunen operator-valued  calculus \cite{Ruz}. So, if $\delta_g\in \mathscr{D}'(G)$ is the Dirac point mass distribution at $g\in G,$ and $D=A\in \Psi^0(G)$ is an elliptic operator on $G,$ we prove that
\begin{equation}\label{CardonaDuvanIF}
\textnormal{ind}(A)=\int_{G}\mu_\gamma(g)dg,\,\,
\end{equation} where \begin{equation}
\mu_\gamma(g):=\exp(-\gamma \sigma_{A^*}(g)\sigma_A(g))\delta_g(g)-\exp({-\gamma \sigma_{A}(g)\sigma_{A^*}(g)})\delta_g(g),
\end{equation}
for all $\gamma>0.$ The right hand side of \eqref{CardonaDuvanIF} is understood in the sense of distributions. The corresponding index theorem for operators of general order will be given in Theorem \ref{generalorder}. As it will be observed, our instrumental tool will be the McKean-Singer lemma
\begin{equation}\label{MSIndexFormula}
\textnormal{ind}(A)=\textnormal{tr}(e^{-tA^*A})-\textnormal{tr}(e^{-tAA^*}),\,\,\,t>0.
\end{equation} In our case, the McKean-Singer lemma implies a local index formula of the form
\begin{equation}
\textnormal{ind}(A)=\int_{G}\mu_0(g)dg
\end{equation}
for some density $\mu_0$ on $G$ defined by certain geometrical invariants (see the classical work \cite{atiyabott2} of Atiyah, Bott and Patody).  In this case, the operator-valued symbolic calculus of Ruzhansky and Turunen simplifies the McKean-Singer formula \eqref{MSIndexFormula} to the expression \eqref{CardonaDuvanIF}.  Certainly, the density $\mu_\gamma$ is defined by the operator valued symbols associated to $A$ and $A^*$ and consequently by the global Fourier analysis associated to every compact Lie group. On the other hand, if $\widehat{G}$ denotes the set of equivalence classes of strongly continuous, irreducible and unitary representations on $G,$ we show that every elliptic operator $A\in \Psi^m(G)$ has Fredholm index given by
\begin{equation}
\textnormal{ind}(A)=\int_{G}\sum_{[\xi]\in \widehat{G}}\dim{(\xi)}\textnormal{Tr}((\mu_\gamma(x)\xi)(e_G))\,dx,
\end{equation}
where
\begin{equation}
(\mu_\gamma(x)\xi)(e_G)=(\exp(-\gamma \sigma_{A^*}(x)\sigma_{A}(x))\xi)(e_{G})-(\exp(-\gamma \sigma_{A}(x) \sigma_{A^*}(x) )\xi)(e_G),
\end{equation}
here $e_G$ is the identity element of $G.$ In this case, $\tilde{\mu}_{\gamma}(x)\xi:= \dim{(\xi)}\textnormal{Tr}((\mu_\gamma(x)\xi)(e_G))$ can be viewed as a density on the non-commutative phase space $G\times \widehat{G}.$  So, the index of elliptic operators on compact Lie groups can be written in terms of the algebraic information  in the representation theory of the group and the operator-valued symbol of these operators. We refer the reader to the references \cite{BerlineVergne,Bott,Goette,Hong,Hong2,Sabin2011,Urakawa} where certain index formulae have been proved for specific elliptic operators on homogeneous spaces $G/K$.

This paper is organized as follows. In Section \ref{Preliminaries} we present some basics on the Fourier analysis used in our context and the global quantization of operators trough of operator-valued symbols. The operator-valued functional calculus will be developed in Section \ref{Operator-valued fucntional calculus}. In Section \ref{MAINTTT} we prove our index formulae.  Finally, in Section \ref{examples} we provide explicit index formulae for elliptic operators on the torus, and the  groups $\textnormal{SU}(2)$ and $\textnormal{SU}(3).$

\section{Preliminaries}\label{Preliminaries}

In this section we present some topics on compact Lie groups,  the Fourier analysis used here, and the operator valued calculus of Ruzhansky and Turunen. For this we follow \cite{Fe1} and \cite{Ruz}. The reference \cite{Hor2} includes a complete background on the theory of pseudo-differential operators.

\subsection{The operator-valued quantization}
Throughout of this paper $G$ is a compact Lie group endowed with its normalised Haar measure $dg$. For $f\in \mathscr{D}'(G),$ the respective right-convolution operator $r(f):C^\infty(G)\rightarrow C^\infty(G)$ is defined by
\begin{eqnarray}
r(f)g=g\ast f,\,\,\,g\in C^{\infty}(G).
\end{eqnarray}
If $f\in L^2(G),$ the (right) global Fourier transform is defined by \begin{equation}
r(f)=\int_{G}f(y)\pi_R(y)^*dy,
\end{equation}
where $\pi_R$ is the right regular representation on $G,$ defined by $\pi_R(x)f(y)=f(yx)$ and $\pi_{R}(x)^*=\pi_{R}(x^{-1}),$  $x\in G.$ In this case, the Fourier inversion formula gives
\begin{equation}
f(x)=\textnormal{tr}(r(f)\pi_R(x)),\,\,f\in C^{\infty}(G), x\in G.
\end{equation}
If $\rho:G\rightarrow \mathscr{B}(C^\infty(G))$ is a continuous operator, the pseudo-differential operator $A$ associated to $\rho,$ is defined by
\begin{equation}
Af(x)=\textnormal{tr}(\rho(x)r(f)\pi_R(x)),\,\,f\in C^{\infty}(G).
\end{equation}

Conversely, if $A:C^\infty(G)\rightarrow C^\infty(G)$ is a continuous linear operator, then there exists an unique $\sigma_{A}:G \rightarrow \mathscr{B}(C^\infty(G))$ (called the operator-valued symbol of $A$) satisfying
\begin{equation}
Af(x)=\textnormal{tr}(\sigma_A(x)r(f)\pi_R(x)),\,\,f\in C^{\infty}(G).
\end{equation} The symbol $\sigma_A$ is defined as follows. Let $K_A\in C^\infty(G)\widehat\otimes\mathscr{D}'(G)$ be the distributional Schwartz kernel of $A$ and $R_{A}(x,y)=K(x,y^{-1}x)$ is the right-convolution kernel associated to $A.$ If  $x\in G,$ and $R_A(x)\in \mathscr{D}'(G)$ is defined by $(R_A(x))(y)=R_A(x,y)$ for every $y\in G,$ the (right) operator-valued symbol $\rho=\sigma_A$ associated to $A$ is defined by 
$
\sigma_A(x):=r(R_A(x)),\,\,\,x\in G.
$
In our further analysis will be useful the following composition theorem.
\begin{proposition}[Composition formula for operator valued symbols]\label{composition}
Let us assume that $A$ and $E$ are continuous linear operators on $C^\infty(G).$ Then, for every $x\in G$ we have
\begin{equation}
\sigma_{AE}(x)=\sigma_{A}(x)\sigma_{E}(x).
\end{equation}
\begin{proof}
Let us note that for $f\in C^\infty(G),$
\begin{equation}
AEf(x)=\textnormal{tr}(\sigma_A(x)r(Ef(\cdot))\pi_R(x)).
\end{equation} Since $$r(Ef(\cdot))=r((f\ast R_E(\cdot))(\cdot))=r(R_E(\cdot))r(f)=\sigma_E(\cdot)r(f),$$ we deduce
\begin{equation}
AEf(x)=\textnormal{tr}(\sigma_A(x)\sigma_E(x)r(f)\pi_R(x)),
\end{equation} and by uniqueness we have $\sigma_{AE}(x)=\sigma_A(x)\sigma_E(x).$ So, we finish the proof.
\end{proof}

\begin{remark}\label{remarkpotencia}     
If $A:C^\infty(G)\rightarrow C^\infty(G) $ is a continuous linear operator, for every $k\in\mathbb{N},$ the  operator  $A^k$ is continuous  on $C^\infty(G),$ and by the preceding result the operator valued symbol of $A^k$ satisfies $\sigma_{A^k}(x)=\sigma_{A}(x)^k,$  $x\in G.$ It is easy to see that if $A$ is invertible on $C^{\infty}(G),$ then $\sigma_{A^{-1}}(x)=\sigma_A(x)^{-1},$ for all $x\in G.$
\end{remark}

\end{proposition}

\subsection{Fredholm operators} The index is defined for a broad class of operators called Fredholm operators. Now,  we introduce this notion in more detail. For $X,Y$ normed spaces $B(X,Y)$ is the set of bounded linear operators from $X$ into $Y.$
\begin{definition}
 If $H_1$ and  $H_2$ are Hilbert spaces, the closed and densely defined operator $A:H_1\rightarrow H_2$ is Fredholm if only if $\text{Ker}(A)$ is finite dimensional and $\textnormal{Im}(A)=\text{Rank}(A)$ is a closed subspace of $H_2$ with finite codimension. In this case, the index of $A$ is defined by $\text{ind}(A)=\dim  \text{Ker}(A)-\dim  \text{Coker}(A).$ The index formula also can be written as
 $$\text{ind}(A)=\text{dim}\text{Ker} (A)-\text{dim}\text{Ker} (A^*).  $$
\end{definition}
 
Now, we end this subsection with  a result now known as McKean-Singer index formula. We present the proof by completeness. 
\begin{lemma}\label{heatapproach}
Let us assume that $T:H_{1}\rightarrow H_{2}$ is a Fredholm operator, $ TT^{*} $ and $T^*T$ have a discrete spectrum, and for all $t>0,$ $ e^{-tT^*T}$ and $e^{-tTT^*}$ are trace class. Then \begin{equation}
\textnormal{ind}(T)=\textnormal{tr}(e^{-tT^*T})-\textnormal{tr}(e^{-tTT^*}).
\end{equation}  
\end{lemma}
\begin{proof}
Let $\lambda \in \Sigma(T^*T),$ $\lambda\neq 0,$ then there exists a non-zero vector $\phi\in H_{1}$ such that $T^*T(\phi)= \lambda \phi.$ Then, $TT^*T(\phi)= \lambda T\phi$ so that if $T(\phi)$ is non-zero, then it is an eigenvalue of $TT^*.$ It follows that the non-zero eigenvalues of $TT^*$ and $TT^*$ are the same. Thus
$$ \textnormal{tr}(e^{-tT^*T})-\textnormal{tr}(e^{-tTT^*})=\dim \text{Ker} (T^{*}T)-\dim \text{Ker} (TT^{*}).$$
Since $\text{Ker}(A^*A)=\text{Ker}(A)$ and $\text{Ker}(AA^*)=\text{Ker}(A^*)$ we end the proof.
\end{proof}  
\subsection{The matrix-valued quantization} Let us consider for every  compact Lie group $G$ its unitary dual $\widehat{G},$ that is the set of continuous, irreducible, and  unitary representations on $G.$  If $[\xi]\in\widehat{ G},$ and $\xi:G\rightarrow U(\mathbb{C}^{d_\xi}),$ the following equalities follow from the Fourier transform on $G$
 $$  \widehat{f}(\xi)=\int_{G}\varphi(x)\xi(x)^*dx\in \mathbb{C}^{d_\xi\times d_\xi},\, f(x)=\sum_{[\xi]\in \widehat{G}}d_{\xi}\text{Tr}(\xi(x)\widehat{f}(\xi)),\,\,x\in G,f\in C^{\infty}(G) ,$$
and the Peter-Weyl Theorem on $G$ implies the Plancherel Theorem on $L^2(G),$
$$ \Vert f \Vert_{L^2(G)}= \left(\sum_{[\xi]\in \widehat{G}}d_{\xi}\text{Tr}(\widehat{f}(\xi)\widehat{f}(\xi)^*) \right)^{\frac{1}{2}}=\Vert  \widehat{f}\Vert_{ L^2(\widehat{G} ) } .$$
\noindent Notice that, since $\Vert A \Vert_{HS}^2=\text{Tr}(AA^*)$, the term within the sum is defined by the Hilbert-Schmidt norm of the matrix $\widehat{f}(\xi)$. Any linear operator $A$ on $G$ mapping $C^{\infty}(G)$ into $\mathcal{D}'(G)$ gives rise to a {\em matrix-valued global (or full) symbol} $\sigma_{A}(x,\xi)\in \mathbb{C}^{d_\xi \times d_\xi}$ given by
\begin{equation}
\sigma_A(x,\xi)=\xi(x)^{*}(A\xi)(x),\,\,(A\xi)(x):=(A\xi_{ij})_{i,j=1}^{d_\xi},\,\,x\in G,\,[\xi]\in\widehat{G},
\end{equation}
which can be understood from the distributional viewpoint. Then it can be shown that the operator $A$ can be expressed in terms of such a symbol as \cite{Ruz}
\begin{equation}\label{mul}Af(x)=\sum_{[\xi]\in \widehat{G}}d_{\xi}\text{Tr}[\xi(x)\sigma_A(x,\xi)\widehat{f}(\xi)],\,\,x\in G. 
\end{equation}

The Hilbert space $L^2(\widehat{G})$ is defined by the norm  $$\Vert \Gamma \Vert^2_{L^2(\widehat{G})}=\sum_{[\xi]\in\widehat{G}}d_{\xi}\Vert \Gamma (\xi)\Vert^2_{HS},\,\,\,\Gamma(\xi)\in \mathbb{C}^{d_\xi\times d_\xi}.$$ 
Now, we want to introduce Sobolev spaces and, for this, we give some basic tools. \noindent Let $\xi\in \textnormal{Rep}(G):=\cup \widehat{G},$ if $x\in G$ is fixed, $\xi(x):H_{\xi}\rightarrow H_{\xi},$  $H_\xi\cong \mathbb{C}^{d_\xi},$ is an unitary operator and $d_{\xi}:=\dim H_{\xi} <\infty.$ There exists a non-negative real number $\lambda_{[\xi]}$ depending only on the equivalence class $[\xi]\in \hat{G},$ but not on the representation $\xi,$ such that $-\mathcal{L}_{G}\xi(x)=\lambda_{[\xi]}\xi(x);$ here $\mathcal{L}_{G}$ is the Laplacian on the group $G$ (in this case, defined as the Casimir element on $G$). Let  $\langle \xi\rangle$ denote the function $\langle \xi \rangle=(1+\lambda_{[\xi]})^{\frac{1}{2}}$.  
\begin{definition}\label{sov} For every $s\in\mathbb{R},$ the {\em Sobolev space} $H^s(G)$ on the Lie group $G$ is  defined by the condition: $f\in H^s(G)$ if only if $\langle \xi \rangle^s\widehat{f}\in L^{2}(\widehat{G})$. 
\end{definition}
The Sobolev space $H^{s}(G)$ is a Hilbert space endowed with the inner product $\langle f,g\rangle_{s}=\langle \Lambda_{s}f, \Lambda_{s}g\rangle_{L^{2}(G)}$, where, for every $r\in\mathbb{R}$, $\Lambda_{s}:H^r\rightarrow H^{r-s}$ is the bounded pseudo-differential operator with symbol $\langle \xi\rangle^{s}I_{\xi}$. In this paper the notion of Sobolev spaces $H^{s}(G)$ is essential. Indeed, every elliptic operator $T\in \Psi^m(G)$ of order $m$ is a bounded operator from $H^{s}(G)$ into $H^{s-m}(G)$   and, more importantly, its index --as an operator from $C^{\infty}(G)$ to $C^{\infty}(G)$-- agrees with the index of $T$ as operator acting from $H^{s}(G)$ into $H^{s-m}(G)$, for every $s\in \mathbb{R}$.
\begin{definition} Let $(Y_{j})_{j=1}^{\text{dim}(G)}$ be a basis for the Lie algebra $\mathfrak{g}$ of $G$, and let $\partial_{j}$ be the left-invariant vector fields corresponding to $Y_j$. We define the differential operator associated to such a basis by $D_{Y_j} = \partial_{j}$ and, for every $\alpha\in\mathbb{N}^n$, the {\em differential operator} $\partial^{\alpha}_{x}$ is the one given by $\partial_x^{\alpha}=\partial_{1}^{\alpha_1}\cdots \partial_{n}^{\alpha_n}$. Now, if $\xi_{0}$ is a fixed irreducible  representation, the matrix-valued {\em difference operator} is the given by $\mathbb{D}_{\xi_0}=(\mathbb{D}_{\xi_0,i,j})_{i,j=1}^{d_{\xi_0}}=\xi_{0}(\cdot)-I_{d_{\xi_0}}$. If the representation is fixed we omit the index $\xi_0$ so that, from a sequence $\mathbb{D}_1=\mathbb{D}_{\xi_0,j_1,i_1},\cdots, \mathbb{D}_n=\mathbb{D}_{\xi_0,j_n,i_n}$ of operators of this type we define $\mathbb{D}^{\alpha}=\mathbb{D}_{1}^{\alpha_1}\cdots \mathbb{D}^{\alpha_n}_n$, where $\alpha\in\mathbb{N}^n$. Other properties on these differences operators can be found in \cite{RuzWirth2015}. See also \cite{Fischer}.
\end{definition}
Now we introduce, for every $m\in\mathbb{R}$, the H\"ormander class $\Psi^{m}(G)$ of pseudo-differential operators of order $m$ on the compact Lie group $G$. As a compact manifold we consider $\Psi^{m}(G)$ as the set of those operators which, in all local coordinate charts, give rise to pseudo-differential operators in the H\"ormander class $\Psi^{m}(U)$ for an open set $U \subset \mathbb{R}^n$, characterized by symbols satisfying the usual estimates \cite{Hor2}
\begin{equation}
|\partial_{x}^{\alpha}\partial_{\xi}^{\beta}\sigma(x,\xi)|\leq C_{\alpha,\beta}\langle \xi\rangle^{m-|\beta|},
\end{equation}
for all $(x,\xi) \in T^*U \cong \mathbb{R}^{2n}$ and $\alpha,\beta\in \mathbb{N}^n$. This class contains, in particular, differential operator of degree $m>0$ and other well-known operators in global analysis such as heat kernel operators.
\\
\\
We reserve the notation $\Psi^{m}(G\times \widehat{G})=:\Psi^{m}(G)$ for pseudo-differential operators of order $m$ on $G$.    The H\"ormander classes $\Psi^{m}(G)$ were characterized in \cite{Ruz} (see also \cite{RuzTurWirth2014}) by the condition: $A\in \Psi^{m}(G)$ if only if its matrix-valued symbol $\sigma_{A}(x,\xi)$ satisfies the inequalities
\begin{equation}
\Vert \partial_{x}^{\alpha}\mathbb{D}^{\beta}\sigma_{A}(x,\xi)\Vert_{op} \leq C_{\alpha,\beta} \langle \xi\rangle^{m-|\beta|},\,\,x\in G,\,[\xi]\in\widehat{G},
\end{equation}
for every $\alpha,\beta\in \mathbb{N}^n.$ For a rather comprehensive treatment of this quantization process we refer to \cite{Ruz}. In this paper we are interested in the index of elliptic operators in $\Psi^{m}(G)$, where $m\in \mathbb{R}.$ Now, we present the following theorem on elliptic pseudo-differential operators.

\begin{theorem}\label{elipticidadlie}
An operator $A\in \Psi^{m}(G)$  is elliptic if and only if its matrix-valued symbol $\sigma_{A}(x,\xi)$ is an invertible matrix for all but finitely many $[\xi]\in \widehat{G},$ and for all such $\xi$ and $x\in G$ satisfies
$$\Vert \sigma_{A}(x,\xi)^{-1}\Vert_{op} \leq C \langle \xi\rangle^{-m}. $$
Thus both statements are equivalent to the existence of $B\in \Psi^{-m}(G)$ such that $R_{1}=I-AB$ and $R_{2}=I-BA$ are smoothing. This means that $R_{i}\in \Psi^{-\infty}(G):=\cap_{m}\Psi^{m}(G), $ for $i=1,2.$ \end{theorem}

\section{Operator-valued functional calculus for operators on compact Lie groups}\label{Operator-valued fucntional calculus}

In this section we develop the operator-valued functional calculus for pseudo-differential operators on compact Lie groups. It is important to mention that the  matrix-valued functional calculus has been developed by M. Ruzhansky and J. Wirth in \cite{RuzWirth2014}. Our starting point  is the following lemma where  the set $\rho(T)$ denotes the resolvent set of $T.$

\begin{lemma}\label{resolventesymbol}
Let us assume that $\omega\in \rho(T),$ and $T:C^\infty(G)\rightarrow C^\infty(G)$ extends to a closed operator on $L^2(G).$ Then for all $x\in G,$ $\omega\in \rho(\sigma_{T}(x))$ and we have the following identity for the resolvent operator of $T.$
\begin{equation}
\sigma_{(\omega-T)^{-1}}(x)=(\omega-\sigma_T(x))^{-1}.
\end{equation}
\end{lemma}
\begin{proof}
If $\omega \in \rho(T),$ then $R_\omega(T):=\omega-T$ is a bounded and invertible operator on $L^2(G)$ and consequently,
\begin{equation}
 (\omega-T)^{-1} (\omega-T) =I=(\omega-T)(\omega-T)^{-1}. 
\end{equation}
By Proposition \ref{composition}, we have
\begin{equation}
\sigma_{ (\omega-T)^{-1}  }(x)\sigma_{ (\omega-T)  }(x)=\sigma_{ (\omega-T)^{-1}  }(x)( \omega-\sigma_{T}(x))=I,  
\end{equation} and 
\begin{equation}
    ( \omega-\sigma_{T}(x))\sigma_{ (\omega-T)^{-1}  }(x)=  \sigma_{ (\omega-T)  }(x)\sigma_{ (\omega-T)^{-1}}(x)=I.
\end{equation} This analysis shows that $R_{\omega}(\sigma_T(x)):=\omega-\sigma_T(x)$ is an invertible operator on $L^2(G)$ and $$\sigma_{ (\omega-T)^{-1}}(x)=( \omega-\sigma_{T}(x))^{-1},$$ for all $x\in G.$ The proof is complete.
\end{proof}

Let us consider  continuous, symmetric  and linear operators $T:C^{\infty}(G)\rightarrow C^{\infty}(G) $ such that $T$ and $\sigma_{T}(x)$ admit    self-adjoint extensions  on $L^2(G),$ for all $x\in G.$ Examples of these operators arise with operators of the form $T=A^*A$ because $\sigma_{T}(x)=\sigma_{A}(x)^*\sigma_{A}(x).$ We keep the same notation for their self-adjoint extensions on $L^2(G)$. We recall that the spectral theorem (for bounded and unbounded self-adjoint operators) implies
\begin{equation}
T=\int\limits_{-\infty}^\infty\lambda dE_{\lambda},\,\,\,\sigma_T(x)=\int\limits_{-\infty}^\infty\lambda dE_{\lambda}(x),\,
\end{equation}
where   $\{E_{\lambda}\}_{-\infty<\lambda<\infty}$ and $\{E_{\lambda}(x)\}_{-\infty<\lambda<\infty}$ are the spectral measures  of $T$ and $\sigma_{T}(x)$ respectively. Now we present the main result of this section. We will use the Stone formula (see Theorem 7.17 of \cite{Weid})
\begin{equation}
E(\lambda)=\lim_{\delta\rightarrow 0^{+} }\lim_{\varepsilon\rightarrow 0^{+}}\int_{-\infty}^{\lambda+\delta}([t-i\varepsilon-T]^{-1}-[t+i\varepsilon-T]^{-1})dt,
\end{equation} in our further analysis.

\begin{theorem}\label{fucntionalformula} Let us assume that $T$ and $\sigma_T(x)$ are as in the previous discussion and $G:(-\infty,\infty)\rightarrow \mathbb{C}$ is a measurable function. Let us define the operator  $G(T)$   by the functional calculus:  $$G(T):=\int\limits_{-\infty}^\infty G(\lambda)dE_\lambda.$$ Then, the identity
\begin{equation}
G(T)f(x):=[G(\sigma_{T}(x))f](x),\,\,\,G(\sigma_{T}(x))=\int\limits_{-\infty}^\infty G(\lambda)dE_\lambda(x), \end{equation} holds true for every $x\in G$ and $f\in \textnormal{Dom}(G(T)).$
\end{theorem}
\begin{proof}
Let $x\in G$ be a fixed coordinate. By using Lemma \ref{resolventesymbol} and the Stone formula, we have
\begin{align*}
    E_\lambda(x) &=\lim_{\delta \rightarrow 0^{+}}\lim_{\varepsilon\rightarrow 0^{+}}\int\limits_{-\infty}^{\lambda+\delta}((t-i\varepsilon-\sigma_{T}(x))^{-1}-(t+i\varepsilon-\sigma_{T}(x))^{-1})dt\\
    &= \lim_{\delta \rightarrow 0^{+}}\lim_{\varepsilon\rightarrow 0^{+}}\int\limits_{-\infty}^{\lambda+\delta}(\sigma_{(t-i\varepsilon-T)^{-1}}(x)-\sigma_{(t+i\varepsilon-T)^{-1}}(x))dt.
    \end{align*} If $f\in C^\infty(G),$ then we deduce,
\begin{align*}
E_\lambda(x)f(x) &=\lim_{\delta \rightarrow 0^{+}}\lim_{\varepsilon\rightarrow 0^{+}}\int\limits_{-\infty}^{\lambda+\delta}(\sigma_{(t-i\varepsilon-T)^{-1}}(x)f(x)-\sigma_{(t+i\varepsilon-T)^{-1}}(x)f(x))dt\\
&=\lim_{\delta \rightarrow 0^{+}}\lim_{\varepsilon\rightarrow 0^{+}}\int\limits_{-\infty}^{\lambda+\delta}((t-i\varepsilon-T)^{-1}f(x)-(t+i\varepsilon-T)^{-1}f(x))dt\\
&=E_{\lambda}f(x),
\end{align*} where in the last line we have used the Stone theorem applied to $T.$ Now, the identity $E_{\lambda}f(x)=E_\lambda(x)f(x)$ implies 
\begin{align*}
    G(T)f(x)&=\int\limits_{-\infty}^{\infty}G(\lambda)dE_\lambda f(x)=\int\limits_{-\infty}^{\infty}G(\lambda)dE_\lambda(x) f(x)=G(\sigma_T(x)) f(x).
\end{align*}
Thus, we finish the proof.
\end{proof}
\begin{remark}
Let us complement the previous theorem by observing that the relation 
\begin{equation}
E_{\lambda}f(x)=E_\lambda(x)f(x),\,\,f\in C^\infty(G),
\end{equation} holds true for all $x\in G.$
\end{remark}

\section{The index  of  operators on compact Lie groups}\label{MAINTTT}

In this section we prove our main result. Since the prototype of Fredholm operators are elliptic operators, we  classify such condition in terms of the operator-valued quantization.

\subsection{Ellipticity in terms of the operator-valued quantization} In terms of the representation theory of a compact Lie group and the notion of operator valued symbol, the ellipticity of operators can be characterized as follows. 

\begin{theorem}\label{elipticidadlie**}
Let $G$ be a compact Lie group and $e_{G}$ its identity element. An operator $A\in \Psi^{m}(G)$ with operator valued symbol $\sigma_A:G\rightarrow \mathscr{B}(C^\infty(G))$  is elliptic if and only if the   matrix-valued function $\sigma_{A}(x)\xi(e_{G})$  is an invertible matrix for all but finitely many $[\xi]\in \widehat{G},$ and for all such $\xi$ and $x\in G$ satisfies
$$\Vert [\sigma_{A}(x)\xi(e_G)]^{-1}\Vert_{op} \leq C \langle \xi\rangle^{-m}. $$
Thus both statements are equivalent to the existence of $B\in \Psi^{-m}(G)$ such that $R_{1}=I-AB$ and $R_{2}=I-BA$ are smoothing. This means that $R_{i}\in \Psi^{-\infty}(G):=\cap_{m}\Psi^{m}(G), $ for $i=1,2.$ 
\end{theorem}
\begin{proof}
Let us denote by $B(x,\xi)$ the matrix-valued symbol associated to $A.$ This means that $B:G\times \widehat{G}\rightarrow \cup_{[\xi]\in\widehat{G}}(\mathbb{C}^{d_\xi})$ satisfies $B(x,\xi)=\xi(x)^*(A\xi)(x).$ Theorem 10.11.16 in \cite{Ruz} gives the identity $B(x,\xi)=\sigma_{\sigma_A(x)}(y,\xi):=\xi(y)^*(\sigma_{A}(x)\xi)(y),$ for all $y\in G.$ In particular, if $y=e_{G}$ is the identity element in $G,$
\begin{eqnarray}
\sigma_A(x)\xi(e_G)=B(x,\xi),\,x\in G,\,\,[\xi]\in \widehat{G}.
\end{eqnarray} Thus, from Theorem \ref{elipticidadlie}
we finish the proof.
\end{proof}
\begin{remark}
A similar analysis as in the previous result gives the following characterization of H\"ormander classes in terms of operator-valued symbols. in fact,   $A\in \Psi^m(G)$ if and only if
\begin{equation}
\Vert \partial_{x}^{\alpha}\mathbb{D}^{\beta}(\sigma_A(x)\xi(e_G))\Vert_{op} \leq C_{\alpha,\beta} \langle \xi\rangle^{m-|\beta|},
\end{equation} for all $\alpha,\beta\in\mathbb{N}^n.$ It is well know that pseudo-differential operators in $\Psi^0(G)$ are bounded operators on $L^2(G)$ (see \cite{Hor2}).
\end{remark}
\subsection{Index formulae for elliptic operators}
Now we prove our main results. We start with the following theorem.
\begin{theorem}\label{THEOREM!INDEX}
Let $G$ be a compact Lie group and $A\in \Psi^0(G)$ be an elliptic operator. Then the analytical index of $A$ is given by
\begin{equation}\label{CardonaDuvanIF'}
\textnormal{ind}(A)=\int_{G}\mu_\gamma(g)dg,\,\,
\end{equation} where \begin{equation}
\mu_\gamma(g):=\exp(-\gamma \sigma_{A^*}(g)\sigma_A(g))\delta_g(g)-\exp({-\gamma \sigma_{A}(g)\sigma_{A^*}(g)})\delta_g(g),
\end{equation}
for all $\gamma>0,$ where $\delta_g$ is the Dirac point mass at $g\in G.$
\end{theorem}
\begin{proof}
Let us assume that $A\in \Psi^0(G)$ is an elliptic operator. Let us denote for all $x\in G,$ by $B_{\sigma_A(x)}$ the matrix-valued symbol associated to $\sigma_A(x):C^\infty(G)\rightarrow C^\infty(G).$ If $B(x,\xi)$ is the matrix-valued symbol associated to $A,$ then Theorem 10.11.16 in \cite{Ruz} gives
\begin{equation}
B_{\sigma_A(x)}(y,\xi)=B(x,\xi),\,\,x,y\in G,[\xi]\in \widehat{G}.
\end{equation}
Consequently for every $x\in G,$ $\sigma_A(x)\in \Psi^0(G)$ is an elliptic operator. So, for  every $x\in G,$ $\sigma_A(x)$ extends to  a bounded and Fredholm operator  on $L^2(G)$ and the operators $A, A^*A, AA^*,$ $\sigma_A(x), \sigma_A(x)^*\sigma_A(x), \sigma_A(x)\sigma_A(x)^*,$ $e^{-\gamma A^*A},$ $e^{-\gamma AA^*},$ $e^{-\gamma\sigma_{A}(x)^* \sigma_{A}(x)  },$ and $e^{-\gamma\sigma_{A}(x) \sigma_{A}(x)^*  }$ have discrete spectrum. In order to compute the index of $A$ we need to compute the operator-valued symbol of the operators  $e^{-\gamma A^* A  }$ and $e^{-\gamma A {A}^*  }.$ Although this can be done by using the operator-valued functional calculus developed in the previous section, we give a more elementary construction.  For this we will use that the
exponential formula $e^T=\sum_{k=0}^{\infty}\frac{1}{k!}T^{k}$ holds true for a bounded operator $T,$ on a Hilbert space $H,$ where we have assumed that $T$ is self-adjoint and $e^T$ is defined by the functional calculus. From Proposition \ref{composition} and Remark \ref{remarkpotencia},   we have for $f\in C^\infty(G)$ and $\gamma>0,$ the identity $e^{-\gamma A^*A}f(x)=e^{-\gamma\sigma_{A^*}(x)\sigma_A(x)}f(x),$ in fact
\begin{align*}e^{-\gamma A^*A}f(x) &=\left(\sum_{k=0}^\infty (-\gamma)^k/k!(A^*A)^k\right)f(x)
=\sum_{k=0}^\infty (-\gamma)^k/k!((A^*A)^k f)(x)\\
& =\sum_{k=0}^\infty (-\gamma)^k/k!\sigma_{(A^*A)^k}(x)f(x)=\sum_{k=0}^\infty (-\gamma)^k/k!(\sigma_{(A^*A)}(x))^k f(x)\\
&=\sum_{k=0}^\infty (-\gamma)^k/k!(\sigma_{A^*}(x)\sigma_A(x))^k(x)f(x)\\
&=e^{-\gamma\sigma_{A^*}(x)\sigma_A(x)}f(x),\,\,x\in G,
\end{align*} where we have used that $A$ and $\sigma_A(x)$ are bounded operators on $L^2(G)$ justifying so the convergence computations with the exponentials operators. So,   the operator valued symbol associated to $e^{-\gamma A^*A}$ is given by $\sigma_{e^{-\gamma A^*A}}(x)=e^{-\gamma \sigma_{A^*}(x)\sigma_A(x)}.$ On the other hand, let us denote $K_{e^{-\gamma A^*A}}$ to the distributional kernel associated to $e^{-\gamma A^*A}.$ Then,
\begin{eqnarray}
\textnormal{tr}(e^{-\gamma A^*A})=\int_{G}K_{e^{-\gamma A^*A}}(g,g)dg.
\end{eqnarray} Because $K_{\gamma}(g,g)=R_{e^{-\gamma A^*A}}(g,e_G),$ for every $g\in G,$ (here $e_G$ is the identity element of $G$) we have the distributional  identity
\begin{equation}
e^{-\gamma A^*A}\delta_g(g)=\sigma_{  e^{-\gamma A^*A}}(g)\delta_g(g)=\int_{G}\delta_g(y)R_{e^{-\gamma A^*A}}(y^{-1}g)dy=R_{ e^{-\gamma A^*A} }(g,e_G).
\end{equation} Taking into account the first part of the proof, we deduce
\begin{eqnarray}
e^{-\gamma \sigma_{A^*}(g)\sigma_A(g)}\delta_g(g)=R_{ e^{-\gamma A^*A} }(g,e_G)
\end{eqnarray} and consequently
\begin{equation*}
\textnormal{tr}(e^{-\gamma A^*A})=\int_{G}{}{e^{-\gamma \sigma_{A^*}(g)\sigma_A(g)}\delta_g(g)}dg.
\end{equation*} Similarly, an analogous analysis applied to $A^*$  instead of $A$ gives
\begin{equation*}
\textnormal{tr}(e^{-\gamma AA^*})=\int_{G}{}{e^{-\gamma \sigma_{A}(g)\sigma_{A^*}(g)}\delta_g(g)}dg.
\end{equation*} So, by Lemma \eqref{heatapproach} we have

\begin{equation}
\textnormal{ind}(A)=\int_{G} ({}{e^{-\gamma \sigma_{A^*}(g)\sigma_A(g)}\delta_g(g)}-{}{e^{-\gamma \sigma_{A}(g)\sigma_{A^*}(g)}\delta_g(g)})dg,
\end{equation} where in the last line we have used Lemma \ref{heatapproach}. So, we finish the proof.
\end{proof}

\begin{remark}
The main advantage here is Proposition \ref{composition} for the composition of operators, where we have a closed formula,  instead of the usual global calculus where is used the notion of asymptotic expansions.
\end{remark}
\begin{remark}
If $A\in \Psi^0(G)$ is a left invariant operator on a compact Lie group, its operator valued symbol is  the constant mapping $\sigma_A(x)=A,$ $x\in G,$ in this case, $A$ is a right convolution operator and from \eqref{CardonaDuvanIF'}, $\textnormal{ind}(A)=0.$ Operators on compact Lie groups with non vanishing index can be found in \cite[Chapter 4]{Ruz}.
\end{remark}

Now, we prove an index theorem for operators of arbitrary order.

\begin{theorem}\label{generalorder}
Let $G$ be a compact Lie group, $m\in\mathbb{R} $ and $A\in \Psi^m(G)$ be an elliptic operator. Then the analytical index of $A$ is given by
\begin{equation}\label{CardonaDuvanIF'''}
\textnormal{ind}(A)=\int_{G}\mu_{\gamma,m}(g)dg,\,\,
\end{equation} where \begin{equation}
\mu_{\gamma,m}(g):=\exp(-\gamma \sigma_{A^*}(g)\Lambda_{-2m}\sigma_A(g))\delta_g(g)-\exp({-\gamma \Lambda_{-m}\sigma_{A}(g)\sigma_{A^*}(g)\Lambda_{-m}})\delta_g(g),
\end{equation}
for all $\gamma>0,$
 where $\delta_g$ is the Dirac point mass at $g\in G.$
\end{theorem}
\begin{proof}
For the proof we apply Theorem \ref{THEOREM!INDEX} to the operator $E=\Lambda_{-m}A\in \Psi^0(G).$ In fact, by using that $\Lambda_{m}$ is self-adjoint, from the logarithmic property of the index we have
$$ \textnormal{ind}(A)=\textnormal{ind}(\Lambda_{m})+\textnormal{ind}(E)=\textnormal{ind}(E). $$
Because $\Lambda_{-m}$ is left invariant, $\sigma_{ \Lambda_{-m} }(x)=\Lambda_{-m},$  and $\sigma_E(x)=\Lambda_{-m}\sigma_A(x).$ Now, $E^*=A\Lambda_{-m},$ $\sigma_{E^*}(x)=\sigma_A(x)\Lambda_{-m}$ and  we have
\begin{equation}\label{CardonaDuvanIF''''''}
\textnormal{ind}(E)=\int_{G}\mu_\gamma(g)dg,\,\,
\end{equation} where \begin{equation}
\mu_{\gamma,m}(g):=\exp(-\gamma \sigma_{A^*}(g)\Lambda_{-2m}\sigma_A(g))\delta_g(g)-\exp({-\gamma \Lambda_{-m}\sigma_{A}(g)\sigma_{A^*}(g)\Lambda_{-m}})\delta_g(g),
\end{equation}
for all $\gamma>0.$ So, we finish the proof.
\end{proof} Now, we need some preliminary results in order to prove our third index theorem.
\begin{proposition}\label{closed}
Let $G$ be a compact Lie group. Every elliptic pseudo-differential operator $T:C^{\infty}(G)\rightarrow C^{\infty}(G)$ of order $m\geq 0$ extends to a closed operator $T$ on $L^2(G).$
\end{proposition}
\begin{proof}
Let us assume that $T\in \Psi^{m}(G)$ is an elliptic operator. We will show that $T$ is closed on $L^2(G).$ Let $f_{n}\rightarrow f$  and assume that $Tf_{n}\rightarrow g$ where the convergence is in the $L^{2}(G)-$norm. We will prove that $Tf=g.$ From the Theorem \ref{elipticidadlie} there exists $S\in \Psi^{-m}(G)$ such that $TS=I+R$ where $R\in \Psi^{-\infty}(G).$ Since operators in $\Psi^{r}(G)$ are bounded on $L^2(G)$ for $r\leq 0,$ we have that $STf_{n}\rightarrow Sg$ and $(ST)f_{n}\rightarrow (ST)f.$ Hence $Sg=STf.$ Notice that $g\in L^{2}(G)$ and $TS(g)=TST(f)=TSTf=Tf+RTf=g+Rg.$ On the other hand
$$ R(Tf)=RT(\lim_{n\rightarrow \infty} f_{n})= \lim_{n\rightarrow \infty} RTf_{n}= R(\lim_{n\rightarrow \infty} Tf_n)=Rg. $$
Now from the equality $Tf+RTf=g+Rg$ we deduce that $Tf=g.$
\end{proof}
The corresponding statement for trace class pseudo-differential operators is the following (for the proof, we follow the approach of the recent works by J. Delgado and M. Ruzhansky \cite{DR,DR1,DR3}).
\begin{theorem}\label{traceondiagonal}
Let $A$ be a pseudo-differential operator on $\Psi^{m}(G),$ $m< -\dim(G)$ with distributional kernel $K(x,y).$ Then $A$ is trace class on $L^2(G)$ and 
$$ \textnormal{tr}(A)=\int_{G}\sum_{[\xi]\in\hat{G}}d_{\xi}\textnormal{Tr}[\sigma_{A}(x,\xi)]dx, $$ where $\sigma_A(x,\xi)$ is the matrix-valued symbol of $A.$
\end{theorem}
\begin{proof}
It is well known that if $A$ is a pseudo-differential operator of order less that $-n=-\dim( G),$  then $A$ is trace class (see \cite{Roe}). Now, the trace $\textnormal{Tr}(A)$ of $A$ is given by $\textnormal{Tr}(A)=\int_{G}K(x,x)dx$ provided that $K(x,y)$ will be  a continuous function on the diagonal. If we assume the boundedness of $x\mapsto K(x,x),$ since, in the case of compact Lie groups we have $K(x,y)=\sum_{[\xi]\in\widehat{G}}d_{\xi}\textnormal{Tr}(\xi(x)\sigma(x,y)\xi(y)^{^*}),$ then, $K(x,x)=\sum_{[\xi]\in\hat{G}}d_{\xi}\textnormal{Tr}[\sigma_{A}(x,\xi)],$ and we could end the proof. So, we only need to show the boundedness of $\kappa(\cdot)=K(\cdot,\cdot).$ 
Let us note that $\kappa$ is a bounded function on $G.$ Indeed, because $A\in \Psi^{m}(G),$ if $I_{d_\xi}$ denotes the identity matrix on $\mathbb{C}^{d_\xi \times d_\xi},$  we have
\begin{align*}
    |\kappa(x)| &\leq \sum_{ [\xi]\in\widehat{G}}d_{\xi}|\textnormal{Tr}[\sigma_{A}(x,\xi)]| = \sum_{   [\xi]\in\widehat{G} }d_{\xi}|\textnormal{Tr}[\sigma_{A}(x,\xi) I_{d_\xi}]|\\
    &\leq  \sum_{   [\xi]\in\widehat{G} }d_{\xi}\Vert \sigma_{A}(x,\xi) \Vert_{\textnormal{HS}}\Vert I_{d_\xi}\Vert_{ \textnormal{HS}}
    =  \sum_{ [\xi]\in\widehat{G}   } d_{\xi}\Vert \sigma_{A}(x,\xi) I_{d_\xi} \Vert_{\textnormal{HS}}\times d_\xi^{\frac{1}{2}}\\
    &\leq \sum_{  [\xi]\in\widehat{G}  } d_{\xi}\Vert  \sigma_{A}(x,\xi) \Vert_{op}\Vert I_{d_\xi} \Vert_{\textnormal{HS}}\times d_\xi^{\frac{1}{2}}=\sum_{  [\xi]\in\widehat{G}  } d_{\xi}^2\Vert  \sigma_{A}(x,\xi)\Vert_{op}\\
    &\lesssim  \sum_{  [\xi]\in\widehat{G}  } d_{\xi}^2\langle \xi\rangle^{m}<\infty,
\end{align*} where we have used that $m<-n$ in order to conclude that the last series converges. So, for  every $\ell\in\mathbb{N},$ let us define the function
\begin{equation}
    \kappa_{\ell}(x):= \sum_{\langle \xi\rangle \leq \ell}d_{\xi}\textnormal{Tr}[\sigma_{A}(x,\xi)].
\end{equation} 
For every $x\in G,$ $|\kappa(x)|<\infty,$ and the sequence of derivable functions $\kappa_\ell,$ $\ell\in\mathbb{N},$ converges pointwise to $\kappa.$  Let
$(Y_{j})_{j=1}^{\text{dim}(G)}$ be a basis for the Lie algebra $\mathfrak{g}$ of $G$, and let $\partial_{j}$ be the left-invariant vector fields corresponding to $Y_j.$ If we assume that  the sequence of continuous functions $\partial_j\kappa_\ell$ converges uniformly to some continuous function $\kappa'$ on $G$ (that is compact) then $\kappa$ is derivable  and $\kappa'=\partial_j\kappa,$ (consequently we could prove the continuity of $\kappa$). So, we need to prove the existence of $\kappa'.$ If $\ell<\ell',$ then 
\begin{equation}
\Vert\partial_j \kappa_{\ell}-\partial_j\kappa_{\ell'}\Vert_{L^{\infty}(G)}\leq  \sum_{  \ell< \langle \xi\rangle \leq \ell' }d_{\xi}|\textnormal{Tr}[\partial_j\sigma_{A}(x,\xi)]|\lesssim \sum_{  \ell<\langle \xi\rangle \leq \ell'  } d_{\xi}^2\langle \xi\rangle^{m}\rightarrow0 , \textnormal{ as }\ell\rightarrow\infty.
\end{equation}So, the sequence of continuous functions $\partial_j\kappa_\ell$ is a Cauchy sequence on  $C(G)$ and we end the proof taking into account that $(C(G),\Vert \cdot \Vert_{L^\infty(G)})$ is a complete normed space.
\end{proof} 

We end this section with the following result.
\begin{theorem}\label{Indexformulagamma}
Let us consider $m\geq 0$ and let $A\in \Psi^m(G)$ be an elliptic operator. Then the index of $A$ is given by
\begin{equation}
\textnormal{ind}(A)=\int_{G}\sum_{[\xi]\in \widehat{G}}d_\xi\textnormal{Tr}((\mu_\gamma(g)\xi)(e_G))\,dg,
\end{equation}
where
\begin{equation}
(\mu_\gamma(g)\xi)(e_G)=(\exp(-\gamma \sigma_{A^*}(g)\sigma_{A}(g))\xi)(e_{G})-(\exp(-\gamma \sigma_{A}(g) \sigma_{A^*}(g) )\xi)(e_G),
\end{equation} for all $\gamma>0.$
\end{theorem}
\begin{proof}
By taking into account that $A$ has discrete spectrum, by Proposition \ref{closed} we can use Theorem \ref{heatapproach} in order to compute the index of $A.$ If $B_{\gamma}(x,\xi)$ is the matrix-valued symbol associated to  $e^{-\gamma A^*A}$ then

\begin{equation}
(\sigma_{e^{-\gamma {A^*}A}}(x)\xi)(e_G)=B_{\gamma}(x,\xi).
\end{equation} If $m>0,$ in order to compute $\sigma_{e^{-\gamma {A^*}A}}(x)$ we  will use the operator-valued functional calculus developed  above. Let us note that for $m=0,$ $A^*A$ is a bounded operator on $L^2(G)$ and the symbol $\sigma_{e^{-\gamma {A^*}A}}(x)$ has been computed in Theorem \ref{THEOREM!INDEX} by using the 
exponential formula $e^T=\sum_{k=0}^{\infty}\frac{1}{k!}T^{k}.$ If $m>0,$ then $A$  could be unbounded on $L^2$ and consequently we cannot use the previous exponential formula. So,  by Theorem \ref{fucntionalformula} applied to $F(t)=e^{-\gamma t}$ and $T=A^{*}A,$ we have the identity $F(A^*A)f(x)=[F(\sigma_{A^*A}(x))f](x),$ $f\in C^\infty(G),$ which in turn is equivalent to
\begin{equation}
    e^{-\gamma A^*A}f(x)=[e^{-\gamma\sigma_{A^*A}(x)}]f(x)=[e^{-\gamma\sigma_{A^*}(x)  \sigma_A(x)}]f(x),\,x\in G,
\end{equation}
where we have used the composition formula $\sigma_{BA}(x)=\sigma_{B}(x)  \sigma_A(x)$ applied to $B=A^*.$  Consequently, by the uniqueness of the operator-valued quantization,  we deduce that
\begin{equation}
\sigma_{e^{-\gamma \sigma_{A^*A } }}(x)=e^{-\gamma\sigma_{A^*}(x)  \sigma_A(x)}, \,\,\textnormal{and,}\,\,B_\gamma(x,\xi)=  (e^{-\gamma\sigma_{A^*}(x)  \sigma_A(x)}\xi)(e_G).
\end{equation}
Now, we have

\begin{equation}
\textnormal{tr}(e^{-\gamma A^*A})=\int_{G}\sum_{[\xi]\in \widehat{G}}d_\xi\textnormal{Tr}(B_\gamma(x,\xi))dx,
\end{equation} where we have used Theorem \ref{traceondiagonal}. So, we have
 \begin{equation}
\textnormal{tr}(e^{-\gamma A^*A})=\int_{G}\sum_{[\xi]\in \widehat{G}}d_\xi\textnormal{Tr}((e^{-\gamma \sigma_{A^*}(x)\sigma_A(x)}\xi)(e_G))dx.
\end{equation} A similar analysis gives 
\begin{equation}
\textnormal{tr}(e^{-\gamma AA^*})=\int_{G}\sum_{[\xi]\in \widehat{G}}d_\xi\textnormal{Tr}((e^{-\gamma \sigma_A(x)  \sigma_{A^*}(x)}\xi)(e_G))dx.
\end{equation} Thus, we obtain
\begin{equation}
\textnormal{ind}(A)=\int_{G}\sum_{[\xi]\in \widehat{G}}d_\xi\textnormal{Tr}((e^{-\gamma \sigma_{A^*}(x)\sigma_A(x)}\xi)(e_G)  - (e^{-\gamma \sigma_A(x)  \sigma_{A^*}(x)}\xi)(e_G) )dx.
\end{equation} With the last line we finish the proof.
\end{proof}
\begin{remark}[Index of matrix-valued pseudo-differential operators] Let us note that  Theorem \ref{heatapproach} implies the following formula for the index of $A.$  
\begin{align}
\text{ind}(A)&=\text{tr}(e^{-\gamma A^{*} A})-\text{tr}(e^{-\gamma AA^{*} }).
\end{align}
We now observe by applying Theorem \ref{traceondiagonal} that
\begin{align*}
\textnormal{ind}(A)&=\text{tr}(e^{-\gamma A^{*} A})-\text{tr}(e^{-\gamma AA^{*} })\\
&=\int_{G}\sum_{ [\xi]\in \hat{G} }d_{\xi}\text{tr}[\sigma_{e^{-\gamma A^{*}A } } (x,\xi)]dx-\int_{G}\sum_{ [\xi]\in \hat{G} }d_{\xi}\text{Tr}[\sigma_{e^{-\gamma AA^{*} } } (x,\xi)]dx\\
&=\int_{G}\sum_{ [\xi]\in \hat{G} }d_{\xi}\text{Tr}[ i_\gamma(x,\xi)  ]dx,
\end{align*} where $i_\gamma(x,\xi):=\sigma_{ [e^{-\gamma A^{*}A }-e^{-\gamma AA^{*} } ]} (x,\xi).$  
\end{remark}

\section{The index of operators on  $\mathbb{T}^n,$  $\textnormal{SU}(2)$ and $\textnormal{SU}(3).$}\label{examples} In this section we consider our index formulae on the $n$-dimensional torus $\mathbb{T}^n$, and the special unitary groups $\textnormal{SU}(2)$ and $\textnormal{SU}(3).$ Our starting point is the following  example in the commutative setting. Our main goal is to  provide explicit computations for the integral expression in Theorem \ref{Indexformulagamma}. 

\begin{example}$(\textnormal{The torus}).$ Let us consider the $n$-dimensional torus $G=\mathbb{T}^n:=\mathbb{R}^n/\mathbb{Z}^n$  and its unitary dual $\widehat{\mathbb{T}}^n:=\{e_\ell:\ell\in\mathbb{Z}^n\},$ $e_\ell(x):=e^{i2\pi\ell\cdot x},$ $x\in\mathbb{T}^n.$ Let us assume that $A\in \Psi^{m}(\mathbb{T}^n),$ $m\geq 0,$ is an elliptic operator.   
Under the identification $\mathbb{T}^n\cong [0,1)^n$, by Theorem \ref{Indexformulagamma}   the index of $A$ is given by
\begin{equation}
\textnormal{ind}(A)=\int\limits_{[0,1)^n}\sum_{\ell\in \mathbb{Z}^n}(\mu_\gamma(x)e^{2\pi i\ell x})(1)\,dx,
\end{equation} where
$$
(\mu_\gamma(x)e^{i2\pi x\ell})(1)=(\exp(-\gamma \sigma_{A^*}(x)\sigma_{A}(1))e^{2\pi i\ell x})(1)-(\exp(-\gamma \sigma_{A}(x) \sigma_{A^*}(x) )e^{2\pi i\ell x})(1),
$$ for all $\gamma>0.$ Other properties on the quantization of global pseudo-differential operators on the torus can be found in \cite{RuzTur2010}.
\end{example}

Now, we present an index formula for elliptic operators on $\textnormal{SU}(2).$ The matrix-valued calculus of pseudo-differential operators for $\textnormal{SU}(2)$ can be found in \cite{RuzTur2013}.

\begin{example}$(\textnormal{The group SU(2)}).$ Let us consider the group $\textnormal{SU}(2)\cong \mathbb{S}^3$ consisting of those orthogonal matrices $A$ in $\mathbb{C}^{2\times 2},$ with $\det(A)=1$.   We recall that the unitary dual of $\textnormal{SU}(2)$ (see \cite{Ruz}) can be identified as
\begin{equation}
\widehat{\textnormal{SU}}(2)\equiv \{ [t_{l}]:2l\in \mathbb{N}, d_{l}:=\dim t_{l}=(2l+1)\}.
\end{equation}
There are explicit formulae for the representations $t_{l}$ as
functions of Euler angles in terms of the so-called Legendre-Jacobi polynomials, see \cite{Ruz}.
Let us consider $m\geq 0$ and let $A\in \Psi^m(\textnormal{SU}(2))$ be an elliptic operator. Then the index of $A$ is given by
\begin{equation}
\textnormal{ind}(A)=\int\limits_{\textnormal{SU}(2)}\sum_{l\in \frac{1}{2}\mathbb{N}}(2l+1)\textnormal{Tr}((\mu_\gamma(g)t_l(g))(I_2))\,dg,
\end{equation}
where $I_2\in \mathbb{C}^{2\times 2}$ is the identity matrix and
\begin{equation}
(\mu_\gamma(g)t_l(g))(I_2)=(\exp(-\gamma \sigma_{A^*}(g)\sigma_{A}(g))t_l(g))(I_2)-(\exp(-\gamma \sigma_{A}(g) \sigma_{A^*}(g) )t_l(g))(I_2),
\end{equation} for all $\gamma>0.$ By using the diffeomorphism 
$\varrho:\textnormal{SU}(2)\rightarrow \mathbb{S}^3,$ defined by
\begin{equation}\varrho(g)=x:=(x_1,x_2,x_3,x_4),\,\,\,\textnormal{for}\,\,\,\,
g=\begin{bmatrix}
    x_1+ix_2       & x_3+ix_4  \\
    -x_3+ix_4       & x_1-ix_2 
\end{bmatrix}, 
\end{equation} we have
\begin{equation}
    \textnormal{ind}(A)=\int\limits_{\mathbb{S}^3}\sum_{l\in \frac{1}{2}\mathbb{N}}(2l+1)\textnormal{Tr}((\mu_\gamma(\varrho^{-1}(x))t_l(\varrho^{-1}(x)))(I_2))\,d\sigma(x),
\end{equation} where $d\sigma(x)$ is the surface measure on $\mathbb{S}^3.$ Now, we want to compute explicitly  the last integral. If we consider the usual parametrization of $\mathbb{S}^3$ defined by $x_1:=\cos(\frac{t}{2}),$ $x_2:=\nu,$ $x_3:=(\sin^2(\frac{t}{2})-\nu^2)^{\frac{1}{2}}\cos(s),$ $x_4:=(\sin^2(\frac{t}{2})-\nu^2)^{\frac{1}{2}}\sin(s),$ where
$$(t,\nu,s)\in D:=\{(t,\nu,s)\in\mathbb{R}^3:|\nu|\leq \sin(\frac{t}{2}),\,0\leq t,s\leq 2\pi\},$$
then $d\sigma(x)=\sin(\frac{t}{2})dtd\nu ds,$ and we have the following index formula in terms of iterated integrals,
\begin{align*}
    &\textnormal{ind}(A) \\&= \int\limits_{0}^{2\pi}\int\limits_{0}^{2\pi}\int\limits_{  -\sin(t/2) }^{ \sin(t/2)   }  \sum_{l\in \frac{1}{2}\mathbb{N}}(2l+1)\textnormal{Tr}((\mu_\gamma(\varrho^{-1}(t,\nu,s))t_l(\varrho^{-1}(t,\nu,s)))(I_2))\,\sin(\frac{t}{2})d\nu dt ds.
\end{align*}\end{example}
We end this section with the following example in the non-commutative context.
\begin{example}
The Lie group $\textnormal{SU}(3)$ (see Fegan \cite{Fe1}) has dimension 8 and $3$ positive square roots $\alpha,\beta$ and $\rho$ with the property
\begin{equation}
\rho=\frac{1}{2}(\alpha+\beta+\rho). 
\end{equation}
We define the weights
\begin{equation}
\sigma=\frac{2}{2}\alpha+\frac{1}{3}\beta,\,\,\,\tau=\frac{1}{3}\alpha+\frac{2}{3}\beta.
\end{equation}
With the notations above the unitary dual of $\textnormal{SU}(3)$ can be identified with
\begin{equation}
\widehat{\textnormal{SU}}(3)\cong\{\lambda:=\lambda(a,b)=a\sigma+b\tau:a,b\in\mathbb{N}_{0}, \}.
\end{equation}
In fact, every representation $\pi=\pi_{\lambda(a,b)}$ has highest weight $\lambda=\lambda(a,b)$ for some $(a,b)\in\mathbb{N}_0^2.$ In this case $d_{{\lambda(a,b)}}:=d_{\pi_{\lambda(a,b)}}=\frac{1}{2}(a+1)(b+1)(a+b+2).$ 
If we choose an elliptic operator $A\in \Psi^m(\textnormal{SU}(3))$ with $m\geq 0,$ then by using Theorem \ref{Indexformulagamma}, we can to compute the index of $A$ by the formula \begin{equation}
\textnormal{ind}(A)=\int\limits_{ \mathbb{SU}(3) }\sum_{a,b\in\mathbb{N}_0}\frac{1}{2}(a+1)(b+1)(a+b+2)\textnormal{Tr}((\mu_\gamma(g)\pi_{a\sigma+b\tau })(I_3))\,dg,
\end{equation}
where $I_3\in \mathbb{C}^{3\times 3}$ is the identity matrix and
\begin{align*}
&(\mu_\gamma(g)\pi_{a\sigma+b\tau })(I_3)\\
&=(\exp(-\gamma \sigma_{A^*}(g)\sigma_{A}(g))\pi_{a\sigma+b\tau })(I_3)-(\exp(-\gamma \sigma_{A}(g) \sigma_{A^*}(g) )\pi_{a\sigma+b\tau })(I_3),
\end{align*} for all $\gamma>0.$ Now, as in the previous example, we compute the integral explicitly. If we consider the  parametrization of $\textnormal{SU}(3)$ (see, e.g., Bronzan \cite{SU3}), $$g\equiv g(\theta,\phi)= g(\theta_1,\theta_2,\theta_3,\phi_1,\phi_2,\phi_3,\phi_4,\phi_5):= (u_{ij})_{i,j=1,2,3}, $$ where $0\leq \theta_i\leq \frac{\pi}{2},\,0\leq \phi_i\leq 2\pi,$ and 
\begin{itemize}
\item $u_{11}=\cos\theta_1 \cos\theta_2 e^{i\phi_1}$
\item $ u_{12}=\sin\theta_1 e^{i\phi_3} $
\item $ u_{13} =\cos\theta_1\sin\theta_2 e^{i\phi_4} $
\item $u_{21} =\sin\theta_1 \sin\theta_3 e^{-i\phi_4-i\phi_5} -\sin\theta_1 \cos\theta_2 \cos\theta_3 e^{i\phi_1+i\phi_2-i\phi_3} $
\item $u_{22} =\cos\theta_1 \cos\theta_3 e^{i\phi_2}   $
\item $u_{23}=-\cos\theta_1\sin\theta_3e^{-i\phi_1-i\phi_5}-\sin\theta_1\sin\theta_2 \cos\theta_3 e^{ i\phi_2-i\phi_3+i\phi_4 }$
\item $u_{31}= -\sin\theta_1\cos\theta_2\sin\theta_3 e^{i\phi_1-i\phi_3+i\phi_5} -\sin\theta_2\cos\theta_3e^{-i\phi_2-i\phi_4}  $

\item $u_{32}=\cos\theta_1\sin\theta_3e^{i\phi_5}$
\item $u_{33} =\cos\theta_2\cos\theta_3 e^{-i\phi_1-i\phi_2}  -\sin\theta_1\sin\theta_2\sin\theta_3e^{-i\phi_3+i\phi_4+i\phi_5}, $
\end{itemize}

then, the group measure is the determinant  given by
\begin{equation}
dg=\frac{1}{2\pi^5}\sin\theta_1 \cos^3\theta_1\sin\theta_2\cos\theta_2\sin\theta_3 \cos\theta_3d\theta_1 d\theta_2 d\theta_3 d\phi_1 d\phi_2 d\phi_3 d\phi_4 d\phi_5,
\end{equation} and we have
 \small{
\begin{align*}
{\textnormal{ind}(A)}
=\frac{1}{4\pi^5}\int\limits_{[0,\frac{\pi}{2}]^3}\int\limits_{[0,2\pi]^5}\sum_{ a,b\in\mathbb{N}_{0}  }(a+b+ab+1)(a+b+2) 
\textnormal{Tr}((\mu_\gamma(g(\theta,\phi))\pi_{a\sigma+b\tau })(I_3))\\
\times   \sin\theta_1 \cos^3\theta_1\sin\theta_2\cos\theta_2\sin\theta_3 \cos\theta_3d\theta\,d\phi
,\end{align*}} with  $(\theta,\phi)=(\theta_1,\theta_2,\theta_3,\phi_1,\phi_2,\phi_3,\phi_4,\phi_5),$ $d\theta=d\theta_1 d\theta_2 d\theta_3,$ and $d\phi=d\phi_1 d\phi_2 d\phi_3 d\phi_4 d\phi_5.$
\end{example}

{ {\bf{ Acknowledgements.}} I would like to thank   Alexander Cardona and Michael Ruzhansky for various discussions.} I also would like to thanks Julio Delgado who has suggested me a gap in the proof of Proposition \ref{traceondiagonal} in a previous version of this document.
The author is indebted with the referee of this paper by his/her asserted suggestions which helped to improve the manuscript.

\bibliographystyle{amsplain}

\end{document}